\documentclass{amsart}     % onecolumn (ditto)

\usepackage{colortbl}
\usepackage{longtable}
\usepackage{hyperref}
\usepackage{amsmath}
\usepackage{amsthm}
\usepackage{pslatex}
\usepackage{amssymb}
\usepackage{latexsym}
\usepackage{graphicx}
\usepackage{multicol}
\usepackage{url}
\usepackage{ulem}
\usepackage{afterpage}
\usepackage{appendix}
\usepackage{tikz}
\usepackage{pst-plot}
\usepackage{pstricks}

\newtheorem{theorem}{Theorem}[section]
\newtheorem*{theorem*}{Main Theorem}
\newtheorem{lemma}[theorem]{Lemma}

\newtheorem{question}{Open Question}
\theoremstyle{definition}

\numberwithin{equation}{section}

\theoremstyle{definition}

\author{Lennard Bakker}
\address{
Mathematics Department\\
Brigham Young University\\
Provo, UT, 84602
}
\email{bakker@math.byu.edu}

\author{Skyler Simmons}
\address{
Mathematics Department\\
Brigham Young University\\
Provo, UT, 84602
}
\email{xinkaisen@gmail.com}

\keywords{Need to fill this in}
\subjclass{Primary ABCDE, Secondary 12345, 67890}
\begin{document}

\title[A Separating Surface]{A Separating Surface for Sitnikov-like $n+1$-body Problems}

\maketitle

\begin{abstract}
We consider the restricted $n+1$-body problem of Newtonian mechanics.  For periodic, planar configurations of $n$ bodies which is symmetric under rotation by a fixed angle, the $z$-axis is invariant.  We consider the effect of placing a massless particle on the $z$-axis.  The study of the motion of this particle can then be modeled as a time-dependent Hamiltonian System.  We give a geometric construction of a surface in the three-dimensional phase space separating orbits for which the massless particle escapes to infinity from those for which it does not.  The construction is demonstrated numerically in a few examples.
\end{abstract}

\section{Introduction}

In 1960, Sitnikov \cite{bibSitnikov1} demonstrated the existence of a restricted three-body problem which exhibits a remarkable chaotic behavior.  This orbit consists of a periodic Kepler two-body orbit in the $xy$ plane with equal masses and a third, massless particle that runs along the $z$ axis, simultaneously remaining not bounded and not tending to infinity (oscillatory motion).  The behavior is eventually extended to the case where the $z$-axis body has small finite mass.  This orbit, and variations of it, have been the subject of much study since that time. \\

Analytic solutions of this problem date back to before Sitnikov's time.  In 1913, MacMillan \cite{bibMacMillan1} gave the explicit solution in the case where the Kepler orbit is circular in terms of elliptic functions.  Finding analytic and numerical methods for computing solutions away from the circular case is still an active area of research (see \cite{bibHagel1}, \cite{bibFaruque1}, and \cite{bibHagelLhotka1}, for example).  Additionally, numerical studies of the non-circular case have given some ideas of qualitative behavior of the orbit, including the nature of bifurcations as the eccentricity parameter varies.  (See \cite{bibJimenezEscalona1}). \\

In 2008 (with further work in 2010) the existence of an infinite family of non-trivial periodic orbits of the Sitnikov problem was demonstrated by Llibre and Ortega in \cite{bibLlibreOrtega1} and Ortega and Rivera in \cite{bibOrtegaRivera1}.  It is demonstrated that there are orbits where $p$ crossings of the massless particle through the origin occur for every $N$ periods of the planar orbit, for any natural numbers $p$ and $N$.  Moreover, it is shown that these periodic orbits exist for any eccentricity in $[0,1)$ for the planar masses.  Similar work was done independently by Marchesin and Castihlo in \cite{bibMarchesinCastilho1}.  Existence results were extended to a generalized Sitnikov problem, which involves more than two masses in a planar configuration whose orbits are ellispses, by Rivera in 2013 (see \cite{bibRivera1}).  Interestingly, the results of Rivera's work included an upper bound of 234 masses in the planar configuration needed for the results to hold. \\

Marchesin and Vidal consider a further departure from the general Sitnikov setting in \cite{bibMarchesinVidal1}.  In their 2013 paper, they consider a rotating central configuration of four masses in two equal-mass pairs which geometrically form a rhombus.  Under a certain transformation of coordinates, the masses become fixed points, as in the classical study of the restricted three-body problem.  They are able to then derive many of the classical results known for the circular Sitnikov problem, including the non-existence of non-periodic oscillating motion.  Additionally, they perform a stability analysis of the so-called ``horizontal motion'', wherein the massless particle is allowed to drift off of the $z$-axis.  This is similar to the numerical stability analysis performed by Sidorenko in \cite{bibSidorenko1} for the circular Sitnikov problem.\\

Some other works, which we will not summarize here, but which lend further evidence to the volume of study given in this topic, include works by Alekseev (\cite{Alekseev1}, \cite{Alekseev2}, \cite{Alekseev3}), Moser (\cite{Moser1}), Liu and Sun (\cite {bibLiuSun1}), Perdios (\cite{bibPerdios1}), Perdios and Markellos (\cite{bibPerdiosMarkellos}), and Soulis, Papadakis, and Bountis (\cite{bibSoulisPapadakisBountis}). \\

The work presented in this paper originally arose out of a study of the rhomboidal four-body problem, in which two bodies of mass $m_1$ lie at $(\pm x, 0)$ for all time, and another pair of mass $m_2$ lie at $(0, \pm y)$.  The periodic version of these orbits feature alternating horizontal and vertical collisions at the origin.  Some relevant recent papers on this orbit include \cite{bibWaldvogel1}, \cite{bibYan1}, and \cite{bibBS1}.  Additionally, \cite{bibMartinez} and \cite{bibShib1} include the rhomboidal four-body orbit as part of a larger class of periodic collision-based orbits.  The $m_2 \to 0$ limiting case of this orbit is equivalent to the eccentricity-one version of the Sitnikov problem, featuring binary collision of the two non-zero masses at the origin.  Since the zero-mass particles are symmetric for all time and exert no gravitational pull on each other, one may be ignored, giving the familiar Sitnikov setting.  The results derived from research in this area readily generalized to the setting presented in this paper. \\

The main portion of the paper will be devoted to the proof of the following theorem:

\begin{theorem*}
There exists a four-branched, two-dimensional topological manifold $\mathcal{S}$ that separates Sitnikov-like $n+1$-body escape orbits from non-escape orbits.  Moreover, each branch of $\mathcal{S}$ is either forward- or backward-invariant.
\end{theorem*}

The remainder of the paper will be as follows: In Section \ref{notation}, we establish the notation that will be used throughout the paper, as well as give the differential equations pertaining to the orbits which we are considering.  Section \ref{helpfulTheorem} contains a theorem from topology that is a key ingredient to the proof of the Main Theorem.  Section \ref{ProofMain} contains the proof of the Main Theorem, broken into three parts.  In Section \ref{proofpart1}, we build up some helpful results for the proof.  Section \ref{proofpart2} constitutes the bulk of the proof, and contains the majority of the construction of $\mathcal{S}$.  Lastly, Section \ref{finishMain} completes the construction and gives some observations about $\mathcal{S}$.  Section \ref{numerics} focuses on numerical results pertaining to the Main Theorem.  Section \ref{preNumerics} gives a few more calculations that can be used to accelerate the pace of the numerical work.  Sections \ref{example1} through \ref{example3} then give the results for various planar configurations.  Lastly, Section \ref{concluding} lists some open questions and gives some concluding remarks. \\

\section{Notation} \label{notation}

Consider any $T$-periodic planar configuration of $n$ bodies whose coordinates are given by $(x_1, y_1)$, ..., $(x_n, y_n)$, and whose masses are given by $m_1$, ..., $m_n$.  We will require that the configuration maintains a rotational symmetry throughout in the following sense: there is a fixed angle $\alpha$ which evenly divides $2\pi$ such that rotation of the plane through the angle $\alpha$ at any time yields the same physical setting (up to re-labeling of the bodies).  For our purposes, no further restrictions need be placed on the planar bodies.  In fact, no difficulty arises if planar orbits featuring regularized collisions are considered.  Under the rotational symmetry condition, the acceleration of a massless particle on the $z$-axis will be in a direction parallel to the $z$-axis.  Moreover, if the initial velocity of the particle on the $z$ axis is also parallel to the $z$-axis, then the particle will remain on the $z$-axis for all time.  \\

As the particular configuration of the planar masses will not be of much importance, we simplify notation slightly by setting
$$r_i = \sqrt{x_i^2 + y_i^2}$$
and consider only the distances from the origin of the bodies in the plane.  Setting $q$ to be the position of the massless particle on the $z$-axis with $p$ its velocity, we have the following equations of motion:
\begin{align}
\label{flow21}\dot{q} &= p, \\
\label{flow22}\dot{p} &= -\sum_{i=1}^n \frac{m_iq}{(r_i^2(t) + q^2)^{3/2}},
\end{align}
where the dot represents the derivative with respect to time.  It is important to remember that each of the $r_i$ is time-dependent and $T$-periodic.  At many points in our analysis, it will help to consider the time-independent system:
\begin{align}
\label{flow1} \dot{q} &= p, \\ 
\dot{p} &= -\sum_{i=1}^n \frac{m_iq}{\left(r_i^2(\theta) + q^2\right)^{3/2}}, \\
\label{flow3} \dot{\theta} &= 1
\end{align}
and consider the behavior on $\mathbb{R} \times \mathbb{R} \times [0,T]$ with the $\theta = 0$ and $\theta = T$ planes identified.  The flow given by equations \ref{flow1} - \ref{flow3} will be denoted $\phi_t$, and points in this space will be given by ordered triples $(p,q,\theta)$.  Note that changing the sign on both $q$ and $p$ also changes the sign of $\dot{q}$ and $\dot{p}$.  Hence, understanding only half of the phase space is necessary to categorize the complete behavior of $\phi_t$.  For simplicity, we will consider the $q > 0$ region.\\

As the behavior near $q = \infty$ will be especially important for our analysis, we also define new variables $Q$ and $P$ by 
\begin{align*}
Q &= q^{-1/2}, \\
P &= p.
\end{align*}
Note that under this change of variables, $Q = 0$ corresponds to $q = \infty$.  In this setting, we have
\begin{align}
\label{FLOW1} \dot{Q} &= -\frac{1}{2}Q^3P,\\
\dot{P} &= -\sum_{i = 1}^n \frac{m_iQ^4}{\left(r_i^2(\theta)Q^4 + 1\right)^{3/2}}, \\
\label{FLOW3} \dot{\theta} &= 1.
\end{align}
We will use $\Phi_t$ to denote the flow on $[0, \infty) \times \mathbb{R} \times [0,T]$ as defined by \ref{FLOW1} - \ref{FLOW3}, where the $\theta = 0$ and $\theta = T$ planes are again identified.  Points in this coordinate setting will again be given by ordered triples $(Q,P,\theta)$. \\

Note that any point $(Q, P, \theta)$ with $Q = 0$ has a $T$-periodic orbit under $\Phi_t$.  These orbits correspond to orbits where the massless particle has escaped to infinity, de-coupling the system.  Physically, the value of $P$ is the velocity with which escape has occurred.  In this setting, the new one-body system is not acted on by external force, and continues moving at its initial velocity in accordance with Newton's first law, while the planar configuration continues its periodic motion forever.  If escape occurs with positive velocity, it is said to be \textit{hyperbolic}.  If escape occurs with zero velocity, it is said to be \textit{parabolic}\\

\section{A Helpful Theorem}\label{helpfulTheorem}

One tool that will be needed in our proof of the Main Theorem, but which is not directly related to the dynamics of the system, is presented below.  It may be thought of as a topological version of the Closed Graph Theorem.  However, to avoid confusion, we will refrain from referring to it as such.  (This is presented as a problem in \ref{bibMunkres}, p. 171.)
\begin{theorem}
Let $f: X \to Y$, where $Y$ is a compact Hausdorff space.  Then $f$ is continuous if and only if the graph $\Gamma_f$ of $f$, defined as
$$\Gamma_f = \{(x, f(x)) \ : \ x \in X\}$$
is closed in $X \times Y$.
\label{closedGraphTheorem}
\end{theorem}

We provide a proof for completeness.

\begin{proof}
Suppose that $f$ is continuous, and let $\{x_n\} \to x$ in $X$.  Then $f(x_n)$ converges to $f(x)$, and so the sequence of points $(x_n, f(x_n))$ converges to $(x, f(x))$.  Since any sequence of points in $\Gamma_f$ corresponds (via projection) to a sequence of points in $x_n$, we get that $\Gamma_f$ is closed. \\

On the other hand, suppose that $\Gamma_f$ is closed.  Let $V$ be any open neighborhood of $f(x_0)$ in $Y$, and let $V^c = Y - V$.  Then $X \times V^c$ is closed, so $\Gamma_f \cap (X \times V^c)$ is closed.  Since $Y$ is compact, projecting the set $\Gamma_f \cap (X \times V^c)$ to $X$ gives a closed set whose points correspond to the points mapped \textit{outside} $V$ by the function $f$.  The complement of this set is therefore the open neighborhood required by the definition of continuity.
\end{proof}

\section{Proof of the Main Theorem}\label{ProofMain}

\subsection{Constructive Lemmas}\label{proofpart1}

In this section, we develop some results that will help to build up the surface described in the Main Theorem.  We begin by making a number of important observations about the flow $\phi_t$.  The first is an observation from calculus.

\begin{lemma}
There exists a positive number $q_\text{mono}$ such that $\dot{p}$ is negative and monotonically increasing as a function of $q$ for all $q > q_\text{mono}$ and for all $\theta$.
\label{preMonotone}
\end{lemma}

\begin{proof}
Recall that $\dot{p}$ is a sum of functions of the form
$$h_i(q) = -\frac{m_iq}{\left(r_i^2(\theta) + q^2\right)^{3/2}}.$$
For a fixed value of $\theta$ with $r_i^2(\theta) > 0$, the function $h_i(q)$ has the shape shown in Figure \ref{accelcurvefigure}.  (The $r_i(\theta) = 0$ case becomes the asymptotic curve $\dot{p} = q^{-3}$.)

\begin{figure}[h]
\begin{center}
\includegraphics[scale=.25]{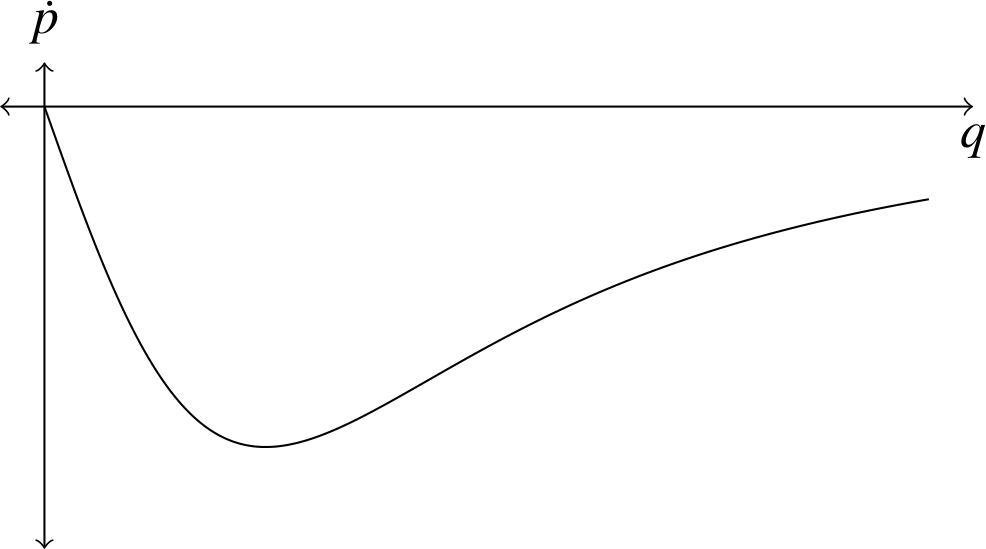}
\caption{Typical shape of $h_i(q)$.}
\label{accelcurvefigure}
\end{center}
\end{figure}

Using basic calculus, we find the single critical point of $h_i$ with $q > 0$ by evaluating $\partial \dot{p} / \partial{q}$.  This value occurs when $q^2 = r_i^2 / 2$, or when $q = r_i/\sqrt{2}$.  (Note that this still holds true for $r_i = 0$.)  Since each $r_i$ is continuous and periodic, each achieves a maximum value $R_i$ over its period.  So the maximum value of $q$ for which the above function can have its critical point is $R_i/\sqrt{2}$.  So $h_i(q)$ is increasing for all $q > R_i/\sqrt{2}$ regardless of the value of $\theta$.  Setting $q_\text{mono}$ to be the maximum of the values $R_i/\sqrt{2}$, then each $h_i(q)$ is increasing for $q > q_\text{mono}$ for any $\theta$.  Since $\dot{p}$ is simply the sum of all the $h_i$, $\dot{p}$ is increasing in $q$ for all $q > q_\text{mono}$.
\end{proof}

Let $\phi_t^q(q_0, p_0, \theta_0)$ represent the value of the $q$ variable under the flow $\phi_t$ with the prescribed initial conditions, and define $\phi_t^p$ and $\phi_t^\theta$ similarly.  It is worth noting that
$$\phi_t^\theta(q_0,p_0,\theta_0) = \theta_0 + t$$
for \textit{any} initial conditions.

\begin{lemma}
With $q_\text{mono}$ as defined in Lemma \ref{preMonotone}, let $q_1$, $q_2$, $p_1$, and $p_2$ be positive numbers with $q_\text{mono} \leq q_1 \leq q_2$ and $p_1 \leq p_2$, and let $\theta_0 \in [0,T]$ be arbitrary.  Let $t_\text{final}$ be the (possibly infinite) maximum value of $t$ for which both $\phi_t^p(q_1, p_1, \theta_0)$ and $\phi_t^p(q_2, p_2, \theta_0)$ are non-negative.  Then, for $t \in (0, t_\text{final})$ we have that both $\phi_t^q(q_1, p_1, \theta_0) \leq \phi_t^q(q_2, p_2, \theta_0)$ and $\phi_t^p(q_1, p_1, \theta_0) \leq \phi_t^p(q_2, p_2, \theta_0)$.
\label{MonotonicityLemma}
\end{lemma}

\begin{proof}
Define $\mathfrak{q} = \phi_t^q(q_2, p_2, \theta_0) - \phi_t^q(q_1, p_1, \theta_0)$ and $\mathfrak{p} = \phi_t^p(q_2, p_2, \theta_0) - \phi_t^p(q_1, p_1, \theta_0)$.  Then, by assumption, both $\mathfrak{q} \geq 0$ and $\mathfrak{p} \geq 0$.  It suffices to show that the set 
$$\{(\mathfrak{q}, \mathfrak{p}) \ : \ \mathfrak{q} \geq 0, \mathfrak{p} \geq 0\}$$
is forward-invariant.  We will do this by showing that the boundary of the region maps to the interior under the flow $\phi_t$.  Note that, by construction, $\dot{\mathfrak{q}} = \mathfrak{p}$.  If $\mathfrak{q} = 0$ and $\mathfrak{p} > 0$, then $\dot{\mathfrak{q}} > 0$.  On the other hand, if $\mathfrak{p} = 0$ and $\mathfrak{q} > 0$, then $\phi_t^q(q_2, p_2, \theta_0) > \phi_t^q(q_1, p_1, \theta_0)$.  This implies that $\dot{\mathfrak{p}} > 0$ when $\mathfrak{p} = 0$ by Lemma \ref{preMonotone}.  Lastly, if both $\mathfrak{q} = \mathfrak{p} = 0$, then $q_1 = q_2$ and $p_1 = p_2$, so $\mathfrak{q} = \mathfrak{p} = 0$ for all time by uniqueness of solution.  Hence, the indicated set is forward-invariant.
\end{proof}

In a physical sense, Lemma \ref{MonotonicityLemma} may be translated as the following: Consider the effect of placing two massless particles on the positive $z$-axis with some upward velocity.  If their initial conditions are not identical and satisfy the conditions of Lemma \ref{MonotonicityLemma}, then:
\begin{itemize}
\item If both start at the same position, then the particle initially moving faster will be \textbf{both} moving faster \textbf{and} located farther away from the origin as long as both continue to move away from the origin.
\item If both start with the same velocity, then the particle initially farther away from the origin will be \textbf{both} located farther away from the origin \textbf{and} moving faster as long as both continue to move away from the origin.
\item If initial positions are not equal, and the particle farther from the origin also has greater velocity, then the particle farther from the origin will be \textbf{both} moving faster \textbf{and} be located farther away from the origin as long as both continue to move away from the origin.
\end{itemize}

Next, we give some analysis of some important behaviors of $\Phi_t$.  For this, we define functions $\Phi_t^Q$ and $\Phi_t^P$ in an analogous fashion to $\phi_t^q$ and $\phi_t^p$.

\begin{lemma}
Let $M = m_1 + \cdots + m_n$.  Then, the set of all points $(Q,P,\theta)$ for which $P \geq \sqrt{2M}Q$ is forward-invariant.
\label{captureconds}
\end{lemma}

Two proofs of this Lemma will be given in this paper.  The first, presented here, is geometric in nature.  A second, more analytic proof is presented as the proof of Theorem \label{numCheck1}.

\begin{proof}
Let $(Q,P,\theta)$ be any point with $P \geq \sqrt{2M}Q$.  Using equations \ref{FLOW1} - \ref{FLOW3} and projecting onto the $QP$-plane, we may think of this region as the area above a line.  (See Figure \ref{figForwardInvariant}.)\\

Now, note that
\begin{align*}
|\dot{P}| &= \sum_{i = 1}^n \frac{m_iQ^4}{\left(r_i^2(\theta)Q^4 + 1\right)^{3/2}} \\
&\leq \sum_{i=1}^n m_iQ^4 \\
&=MQ^4 \\
&=\left(\sqrt{\frac{M}{2}}Q^3\right)\left(\sqrt{2M}Q\right) \\
&\leq \sqrt{\frac{M}{2}}Q^3P.
\end{align*} 
Then, we have that
\begin{align*}
\left|\frac{\partial P}{\partial Q} \right| &= \left|\frac{\dot{P}}{\dot{Q}}\right| \\
&\leq \frac{\sqrt{\frac{M}{2}}Q^3P}{\frac{1}{2}Q^3P} \\
&= 2\sqrt{\frac{M}{2}} \\
&= \sqrt{2M}.
\end{align*}

Geometrically, $\partial P/\partial Q$ represents the slope of a line in the $P, Q$ plane.  In our particular setting, this represents the directions that a trajectory of $\Phi_t^Q$ and $\Phi_t^P$ can take as $t$ increases.  Since both $\Phi_t^Q$ and $\Phi_t^P$ are decreasing, such trajectories must be decreasing in both variables.  Moreover, since the maximum slope that such a trajectory can have is $\sqrt{2M}$ and the line $Q = 0$ consists entirely of equilibria, it is impossible for a trajectory that begins in the $P \geq \sqrt{2M}Q$ region to escape it, as it cannot approach $P = \sqrt{2M}Q$.

\begin{figure}[h]
\begin{center}
\includegraphics[scale=.25]{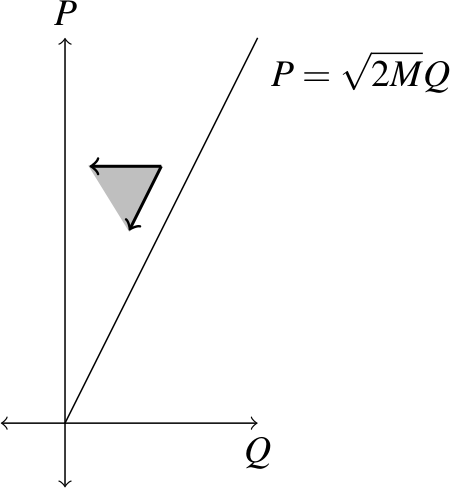}
\caption{The ``fan'' of possible directions of a trajectory of $Q$ and $P$ at a given point.}
\label{figForwardInvariant}
\end{center}
\end{figure}
\end{proof}

\subsection{The Core Construction}\label{proofpart2}

For this section, let $q_0 > q_\text{mono}$ be fixed, where $q_\text{mono}$ is defined as in Lemma \ref{preMonotone}, and let $Q_0 = q_0^{-1/2}$ be the corresponding value in the inverted coordinate frame.  The key step in the construction will be the following:

\begin{theorem}
There exists a continuous function $f(\theta):[0, T] \to \mathbb{R}$ such that 
$$\Phi_t^P(Q_0, f(\theta), \theta) \to 0 \text{ as } t \to \infty.$$
\end{theorem}

The remainder of this section will be the proof of this theorem.  The function $f$ will arise from the level set of another function $g$, which has to be defined in a piecewise fashion.  Figure \ref{constructS} will help to keep much of the notation straight. \\

\begin{figure}[h]
\begin{center}
\includegraphics[scale=.35]{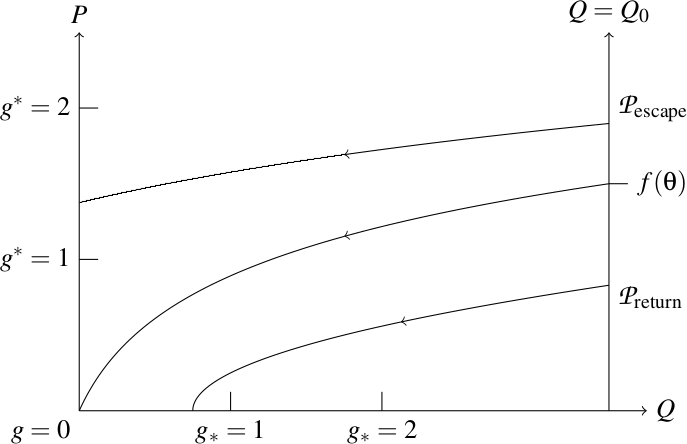}
\caption{Simplified diagram for much of the notation in Section \ref{proofpart2}.  The $\theta$ variable has been removed for ease of reading, but trajectories should be considered as taking place in $(Q,P,\theta)$ space.}
\label{constructS}
\end{center}
\end{figure}

To begin with, define the sets
$$\mathcal{P}_\text{return} = \{(Q_0, P, \theta) \ : \ P \geq 0, \ \theta \in [0, T], \ \Phi_t^P(Q_0, P, \theta)<0 \text{ for some finite }t \geq 0\}$$
and
$$\mathcal{P}_\text{escape} = \{(Q_0, P, \theta) \ : \ P \geq 0, \ \theta \in [0, T], \ \lim_{t \to \infty} \Phi_t^P(Q_0, P, \theta)) > 0\}.$$

Physically, the set $\mathcal{P}_\text{return}$ corresponds to initial conditions that cause the massless particle to return to the origin, as the velocity $P$ eventually becomes negative.  On the other hand, $\mathcal{P}_\text{escape}$ corresponds to initial conditions that lead to the massless particle to escape to infinity with positive velocity (or hyperbolic escape).  Certainly both sets are non-empty, as $(Q_0, 0, \theta) \in \mathcal{P}_\text{return}$ for any $\theta \in [0, T]$, and $\mathcal{P}_\text{escape}$ contains points for which Lemma \ref{captureconds} applies.  Also, if $0 \leq P_1 < P_2$, then by Lemma \ref{MonotonicityLemma}:
\begin{itemize}
\item If $(Q_0, P_1, \theta) \in \mathcal{P}_\text{escape}$, then $(Q_0, P_2, \theta) \in \mathcal{P}_\text{escape}$, and
\item If $(Q_0, P_2, \theta) \in \mathcal{P}_\text{return}$, then $(Q_0, P_1, \theta) \in \mathcal{P}_\text{return}$. \\
\end{itemize}

As the first step in building $g$, for all elements of $\mathcal{P}_\text{return}$, define
$$\tau(Q_0, P, \theta) = \max \{t \geq 0 \ : \ \Phi_t^P(Q_0,P,\theta) \geq 0\},$$
and define $g_*:\mathcal{P}_\text{return} \to \mathbb{R}$ by
$$g_*(Q_0,P,\theta) = \Phi_{\tau(Q_0,P,\theta)}^Q(Q_0,P,\theta).$$
In other words, $g_*$ gives the position (in the inverted coordinate frame) at which the massless particle achieves its maximum before returning to the origin.  Since $\tau$ is continuous by continuity with respect to initial conditions, then $g_*$ is also continuous by composition.  \\

By construction, the function $g_*$ can take values only in the range $(0, Q_0]$.  Moreover, by considering the backwards-time flow $\Phi_{-t}(Q,0,\theta)$ over the set of points where $Q \in (0,Q_0]$ and $\theta \in [0,T]$, it is apparent that $g_*$ is onto.  Since $\mathcal{P}_\text{return}$ is the pre-image of the relatively open set set $(0,Q_0]$, then $\mathcal{P}_\text{return}$ is relatively open in the $Q = Q_0, P \geq 0$ plane.  In fact, $g_*^{-1}((0,Q_0))$ is precisely $g_*^{-1}((0,Q_0])$ with the line $Q = Q_0$, $P = 0$ removed, and has two boundary curves -- the aforementioned $Q = Q_0$, $P = 0$ line, and an upper ($P > 0$) yet-undetermined boundary.  \\

We wish to extend this to a continuous function on $\overline{\mathcal{P}_\text{return}}$, where $\overline{\mathcal{P}_\text{return}}$ denotes the closure of $\mathcal{P}_\text{return}$.  For any sequence of points in $\mathcal{P}_\text{return}$ that approach the open $P > 0$ boundary of $\mathcal{P}_\text{return}$, the sequence must eventually lie outside of $g_*^{-1}((\epsilon, Q_0])$ for any $\epsilon > 0$.  Hence, defining $g_*$ to be zero on the open boundary of $\mathcal{P}_\text{return}$ is a continuous extension of $g_*$ to $\overline{\mathcal{P}_\text{return}}$. \\

We define a similar function $g^* : \mathcal{P}_\text{escape} \to \mathbb{R}$ by
$$g^*(Q_0, P, \theta) = \lim_{t \to \infty} \Phi_t^P(Q_0, P, \theta).$$
Physically, this function describes the velocity with which the massless particle escapes to infinity.  By uniqueness of solutions to ODEs, this function is well-defined. 

\begin{lemma}
The function $g^*$ as just defined is continuous and onto $(0,\infty)$.
\end{lemma}

\begin{proof}

Recall from the proof of Lemma \ref{captureconds} that there is a limited interval of directions (thought of as slopes of lines) that trajectories in the $(Q,P)$ plane can take under $\Phi_t$.  Using this, construct a truncated open ``cone'' $\mathcal{C}$ in the $(Q,P)$ plane as pictured in Figure \ref{gSupStarCts}.  Then, for any point $(Q, P, \theta)$ with $(Q, P) \in \mathcal{C}$, we know that $\Phi_t(Q,P,\theta) \in \mathcal{C} \times [0, T]$ for all $t > 0$.  Let $U$ be the union of all $\Phi_{-t}(\mathcal{C} \times [0,T])$ for $t \geq 0$.  Then the set of all points in $U$ with first coordinate $Q_0$ forms the relatively open set in $\mathcal{P}_\text{escape}$ required for the definition of continuity. \\

\begin{figure}[h]
\begin{center}
\includegraphics[scale=.35]{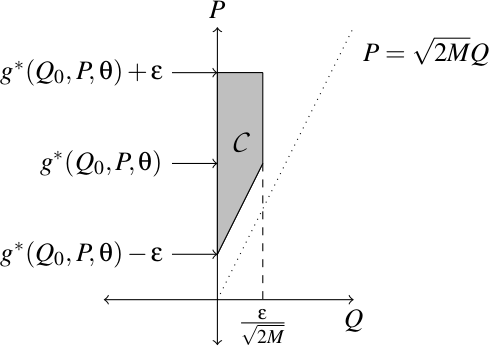}
\caption{Construction of the truncated ``cone'' $\mathcal{C}$.}
\label{gSupStarCts}
\end{center}
\end{figure}

To show that $g^*$ is onto, let $L$ be any positive real number.  Construct two non-intersecting ``cones'' $\mathcal{C}_1$ and $\mathcal{C}_2$ as in Figure \ref{gSupStarCts}, with a sufficiently small $\epsilon$ so that neither borders the point $(0,L)$, and so that $\mathcal{C}_1$ lies below the line $P = L$ and $\mathcal{C}_2$ lies above the line $P = L$.  (See Figure \ref{gSupStarOnto}.)  Then the pre-image of both $\mathcal{C}_i \times [0,T]$ under $\Phi_t$ intersecting $\mathcal{P}_\text{escape}$ contain two points which map to values of $g^*$ which are above and below $L$.  By continuity of $g^*$, there must then be a point at which $g^* = L$.
\end{proof}

\begin{figure}[h]
\begin{center}
\includegraphics[scale=.35]{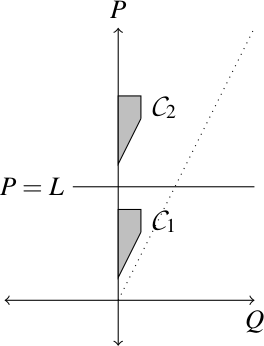}
\caption{The function $g^*$ is onto.}
\label{gSupStarOnto}
\end{center}
\end{figure}

Similar to before, we now note that $\mathcal{P}_\text{escape}$ is open, as it is the continuous pre-image of the set $(0,\infty)$.  We extend $g^*$ to $\overline{\mathcal{P}_\text{escape}}$ by defining $g^*$ to be zero on the boundary points.  Such an extension is continuous.  Finally, let $\mathcal{P} = \overline{\mathcal{P}_\text{return}} \cup \overline{\mathcal{P}_\text{escape}}$, and define $g:\mathcal{P} \to \mathbb{R}$ by
\begin{equation*}
   g(Q_0,P,\theta) = \left\{
     \begin{array}{lr}
       g^*(Q_0,P,\theta) & : (Q_0,P,\theta) \in \overline{\mathcal{P}_\text{escape}}\\
       -g_*(Q_0,P,\theta) & : (Q_0,P,\theta) \in \overline{\mathcal{P}_\text{return}} \\
     \end{array}
   \right.
\end{equation*}
It now remains to show that $\mathcal{P}$ is the set of all points with $Q = Q_0$ and $P \geq 0$.

\begin{lemma}
For $Q = Q_0$ and a given $\theta_0$, there is a single value of $P$ satisfying $g(Q_0, P, \theta_0) = 0$.
\label{singletonLemma}
\end{lemma}

\begin{proof}
Suppose there are two values $P_1$ and $P_2$ with $P_1 < P_2$ so that
$$\Phi_t^P(Q_0, P_i, \theta_0) \to 0 \text{ as } t \to \infty$$
for $i = 1$, $2$.  Then, by Lemma \ref{MonotonicityLemma} and the fact that both are assumed to escape to infinity, we must have $P_1 < P_2$ for all $t > 0$.  Furthermore, by the proof of Lemma \ref{MonotonicityLemma}, the difference $P_2 - P_1$ must be increasing for all $t$.  Thus, it is impossible for 
$$\lim_{t \to \infty} \Phi_t^P(Q_0, P_1, \theta_0) = 0 = \lim_{t \to \infty} \Phi_t^P(Q_0, P_2, \theta_0)$$
for $P_1 \neq P_2$.
\end{proof}

As a consequence, the upper boundary of $\mathcal{P}_\text{return}$ and the boundary of $\mathcal{P}_\text{escape}$ must be the same, as it is impossible for them to have an interval of any positive length separating them for any fixed value of $\theta$.  We can now define $f(\theta)$ to be the unique value such that $g(Q_0, f(\theta), \theta) = 0$.  By construction of $g$, this gives the property
$$\Phi_t^P(Q_0, f(\theta), \theta) \to 0 \text{ as } t \to \infty$$
automatically.  Since both of $g_*$ and $g^*$ are continuous in their respective domains, and they both take the value of $0$ on their shared boundary, the new function $g$ is continuous.  

\begin{lemma}
The function $f$ is well-defined and continuous.
\end{lemma}

\begin{proof}

We know $f$ is well-defined by Lemma \ref{singletonLemma}.  Note that the set of all points $(Q_0, f(\theta), \theta)$ is closed, as it is precisely the set $g^{-1}(0)$.  Then, by Theorem \ref{closedGraphTheorem}, we have that $f$ is continuous, as the range of $f$ is the compact interval $[0, \sqrt{2M}Q_0]$.
\end{proof}

\subsection{The Final Parts}\label{finishMain}

Here, we complete the proof of the Main Theorem and give some properties of the resulting surface $\mathcal{S}$.  Let $\mathcal{G}$ be the set of all points $(Q_0, f(\theta), \theta)$, with $Q_0$ as defined in the previous section.  Then the image of $\mathcal{G}$ under $\Phi_t$ for $t \in \mathbb{R}$ gives a topological 2-manifold $\mathcal{S}$ which is $\Phi_t$-invariant and lies in the $q > 0$, $p > 0$ portion of phase space.  Since changing the signs of $q$ and $p$ changes the sign of $\dot{q}$ and $\dot{p}$ in \ref{flow21} and \ref{flow22}, the surface $\mathcal{S}$ is mirrored in the $q < 0$, $p < 0$ region ``for free''.  Using the same coordinate changes for $Q$ and $P$ and running time backwards, we can construct an analogous surface $\mathcal{S}$ in the $q > 0$, $p < 0$ region and use symmetry to get the remaining two portions. \\

\begin{theorem}
\label{boundedp}
If $r_i(t) \neq 0$ for all $i$ and for all $t$, then the $p$-coordinates of all points in $\mathcal{S}$ are bounded.  In other words, $\mathcal{S}$ lies between the two planes $p = \pm B$ for some $0 < B < \infty$.
\end{theorem}

\begin{proof}
Without loss of generality, consider the portion of $\mathcal{S}$ lying in the region $q > 0$, $p > 0$.  (The $q > 0$, $p < 0$ region follows by reversing time.  The $q < 0$ region then follows from symmetry.)  Note that for any point $x = (q_0, p_x, \theta_x)$ in the set $\mathcal{G}$ described at the beginning of the section, forward images of $x$ under $\phi_t$ will have decreasing values of $p$.  Hence, only the backwards-time image of $x$ need be considered to find a bound.  In that case, the value of $\phi_{-t}^p(q_0, p_x, \theta_x)$ is increasing in $t$, but this rate of increase is bounded above by
$$\sum_{i=1}^n \frac{m_iq_0}{\sqrt{\epsilon + q^2}} = \frac{Mq_0}{\sqrt{\epsilon + q^2}},$$
where $\epsilon > 0$ is the minimum value of all the $r_i$ on $[0,T]$.  Furthermore, since $p_x > 0$, it must reach the $q = 0$ plane before $t = q_0/p_x$.  Hence, the value of $\phi_{-t}^p(q_0, p_x, \theta_x)$ must be finite at the time the trajectory intersects the $q = 0$ plane, as it has a bounded rate of increase over a bounded time.  Since mapping $\mathcal{G}$ to the $q = 0$ plane by $\phi_{-t}$ is continuous, the values of $p$ on the intersection of $\mathcal{S}$ and $\{q = 0\}$ must be bounded.
\end{proof}

It is important to note that the $q > 0$ and $q < 0$ branches of $\mathcal{S}$ need not match up across the $q = 0$ plane.  However, we can guarantee the existence of a few points where the two will match if the planar orbit satisfies certain symmetry properties.

\begin{theorem}
Let $t_0$ be a real number for which all $r_i(t)$ satisfy $r_i(t_0 - t) = r_i(t_0 + t)$.  Then there exists an orbit that escapes to infinity parabolically in both forward and reverse time which passes through the $q = 0$ plane when $\theta = t_0 \text{ mod } T$.
\end{theorem}

\begin{proof}
Let $(0, p_0, \theta_0)$ be a point with $\theta_0 = t_0 \text{ mod } T$ such that $\phi_t(0, p_0, \theta_0) \in \mathcal{S}$ some for $t > 0$.  Then $\phi_t(0, p_0, \theta_0) \in \mathcal{S}$ for all $t > 0$ by flow-invariance of $\mathcal{S}$.  Since $r_i(t_0 - t) = r_i(t_0 + t)$, we then have $\phi_{-t}^q(0, p_0, \theta_0) = -\phi_t^q(0, p_0, \theta_0)$ and $\phi_{-t}^p(0, p_0, \theta_0) = -\phi_t^p(0, p_0, \theta_0)$.  Since $\lim_{t \to \infty} \phi_t^q(0, p_0, \theta_0) = \infty$ and $\lim_{t \to \infty} \phi_t^p(0, p_0, \theta_0) = 0$ by construction, then we must have $\lim_{t \to \infty} \phi_{-t}^q(0, p_0, \theta_0) = -\infty$ and $\lim_{t \to \infty} \phi_{-t}^p(0, p_0, \theta_0) = 0$.
\end{proof}

In the case of the circular Sitnikov problem, such a time-reversing symmetry for the functions $r_i$ exists for all $t_0$, and so the two branches of $\mathcal{S}$ will match exactly over $q = 0$.  This is not too surprising, as this problem is completely integrable.  For configurations with up to finitely many symmetries, the regions on which the two branches of $\mathcal{S}$ fail to line up on $q = 0$ give the ``windows'' through which any orbit passing from return to escape must pass.  These will be the subject of the numerical investigations in the following section. \\

\section{Numerical Investigations}\label{numerics}

\subsection{A Numerical Catalyst}\label{preNumerics}

Before presenting any particular numerical worked examples, we present some theory that makes our numerical work simpler.  Since $0 \leq r_i \leq R_i$ for all applicable $i$, then for $q > 0$, we have
$$q^3 \leq (r_i^2 + q^2)^{3/2} \leq (R_i^2 + q^2)^{3/2}.$$
Let $\mathfrak{R}$ be the maximum of $\{R_1, R_2, ... , R_n\}$.  Then
$$q^3 \leq (r_i^2 + q^2)^{3/2} \leq (\mathfrak{R}_i^2 + q^2)^{3/2}.$$
Taking the reciprocal and multiplying by $m_iq$ gives
$$\frac{m_i}{q^2} \geq \frac{m_iq}{(r_i^2 + q^2)^{3/2}} \geq \frac{m_iq}{(\mathfrak{R}^2 + q^2)^{3/2}}.$$
Assuming further that $p > 0$, we then have that
$$\frac{-m_ip}{q^2} \leq \frac{-m_iqp}{(r_i^2 + q^2)^{3/2}} \leq \frac{-m_iqp}{(\mathfrak{R}^2 + q^2)^{3/2}}.$$
Since this holds for arbitrary $i$, it holds in the summation.  We then have
$$\frac{-Mp}{q^2} \leq \sum_{i = 1}^n \frac{-m_iqp}{(r_i^2 + q^2)^{3/2}} \leq \frac{-Mqp}{(\mathfrak{R}^2 + q^2)^{3/2}}.$$
The central quantity here is simply $p\dot{p}$.  All of these expressions can be integrated explicitly with respect to $t$.  Integrating over the interval $[t_a,t_b]$ gives
\begin{equation}
\frac{M}{q(t_b)} - \frac{M}{q(t_a)} \leq \frac{1}{2}p^2(t_b) - \frac{1}{2}p^2(t_a) \leq \frac{M}{\sqrt{\mathfrak{R}^2 + q^2(t_b)}} - \frac{M}{\sqrt{\mathfrak{R}^2 + q^2(t_a)}}.
\label{EInequalities}
\end{equation}

\begin{theorem}
If $q > 0$, $p > 0$, and 
$$E^* \frac{1}{2}p^2(t_a) - \frac{M}{q(t_a)} > 0$$
at some time $t_a$, then $q(t) \to \infty$ as $t \to \infty$.
\label{numCheck1}
\end{theorem}

\begin{proof}
Re-arranging the right inequality in \ref{EInequalities} gives
$$\frac{1}{2}\left(p(t_b)\right)^2 \geq \frac{M}{q(t_b)} + E^*.$$
Hence, $p$ is bounded below for all time $[t_a, t_b]$.  As $t_b$ was arbitrary and $E^*$ only depends upon conditions at $t_a$, we have that $p$ is bounded below for all time.  Furthermore, $p$ is bounded below uniformly by the value of $E^*$.  Hence, $q$ is increasing with its derivative bounded away from zero, so we must have $q \to \infty$ as $t \to \infty$.
\end{proof}

It is worth noting that in the inverted coordinates, the conditions of Theorem \ref{numCheck1} become precisely those of Lemma \ref{captureconds}.  The geometric proof presented following Lemma \ref{captureconds} was useful for many of the details in Section \ref{ProofMain}.  On the other hand, the proof of Theorem \ref{numCheck1} can be more readily interpreted in terms of positions and velocities, giving a better physical intuition. \\

On the other hand, we also have

\begin{theorem}
If $q > 0$, $p > 0$, and 
$$E_* = \frac{1}{2}p^2(t_a) - \frac{M}{\sqrt{\mathfrak{R}^2 + q^2(t_a)}} < 0$$
at some time $t_a$, then there is some future time $t_b$ where $p(t_b) = 0$ and $q(t_b) < \infty$.
\label{numCheck2}
\end{theorem}

\begin{proof}
Rearranging the left inequality in \ref{EInequalities}, we have
$$\frac{1}{2}p^2(t_a) \leq \frac{M}{\sqrt{\mathfrak{R}^2 + q^2(t_b)}} + E_*.$$
Since the term on the left must be positive, the term on the right must be as well.  This gives
$$-E_* \leq \frac{M}{\sqrt{\mathfrak{R}^2 + q^2(t_b)}}.$$
(Note that both sides here are positive.)  Rearranging the equation then gives
$$q^2(t_b) \leq \left(\frac{M}{E_*}\right)^2 - \mathfrak{R}^2.$$
Since both numbers on the right are finite, we must have that $q^2$ is bounded, and so $q$ is bounded as long as $q$ and $p$ are both positive.  Then $\ddot{q}$ is negative and bounded away from zero on that same time interval, and so it must be the case that $\dot{q} = p = 0$ at some future time $t_b$.
\end{proof}

It is worth noting that if the inverted coordinates are used, the set $E_* = 0$ becomes the equation
$$P = \frac{\sqrt{2M}Q}{\sqrt[4]{R^2Q^4 + 1}}.$$
As $Q \to 0^+$, the derivative $dP/dQ$ approaches $\sqrt{2M}$, which is the same as the boundary of the forward-invariant region given in Lemma \ref{captureconds}.  Since the invariant surface $\mathcal{S}$ must lie between the two, we can describe the linearized behavior of $\mathcal{S}$ near $Q = 0$--namely, $\mathcal{S}$ is locally approximated by $P = \sqrt{2M}Q$, independent of $\theta$.  Hence, we have a first-order approximation for a stable manifold of a degenerate fixed point (in the two-variable time-dependent system), similar to the work done by McGehee in \cite{bibMcGehee1}. \\

Theorems \ref{numCheck1} and \ref{numCheck2} give readily verifiable conditions on $q$ and $p$ that determine whether the massless particle escapes to infinity or has a future point at which it again passes through the origin.  The set of points in $(q,p,\theta)$ that satisfy these inequalities are not complements of each other in $\mathbb{R}^3$, and the separating surface $\mathcal{S}$ must lie between them.  It is also worth noting that these sets of points are merely forward-invariant under $\phi_t$ -- there are points lying outside of these sets that eventually enter them under $\phi_t$.  This is helpful numerically, as we can obtain estimates of the intersection of $\mathcal{S}$ with the plane $q = q_0 \geq q_\text{mono}$, and hence find the value of $f(\theta)$, by fixing a value of $\theta$ and integrating initial conditions for various estimated values of $p$.  Depending on which region they enter, we can adjust our guess upward or downward in a standard fashion (for instance, using the bisection method).  This is very readily accomplished numerically in the inverted coordinate frame, as the intervals in which $Q$, $P$, and $\theta$ lie are all of finite length.  Lastly, as described in Section \ref{finishMain}, we can integrate to find the intersection of the image of this curve under $\phi_{-t}$ with the plane $q = 0$.  These images can reveal some other interesting possible behaviors of the orbit of the massless particle.\\

\subsection{A Non-circular Kepler Configuration}\label{example1}
As a first example, we consider the classical Sitnikov problem with a non-circular Kepler orbit in the plane.  Two bodies of mass 1 are initially placed at $(\pm 1, 0)$ with initial velocity $(0,\pm 1)$.  This results in the intersecting ellipses shown in Figure \ref{figKeplerBase}.

\begin{center}
\begin{figure}[h]
\includegraphics[scale=.5]{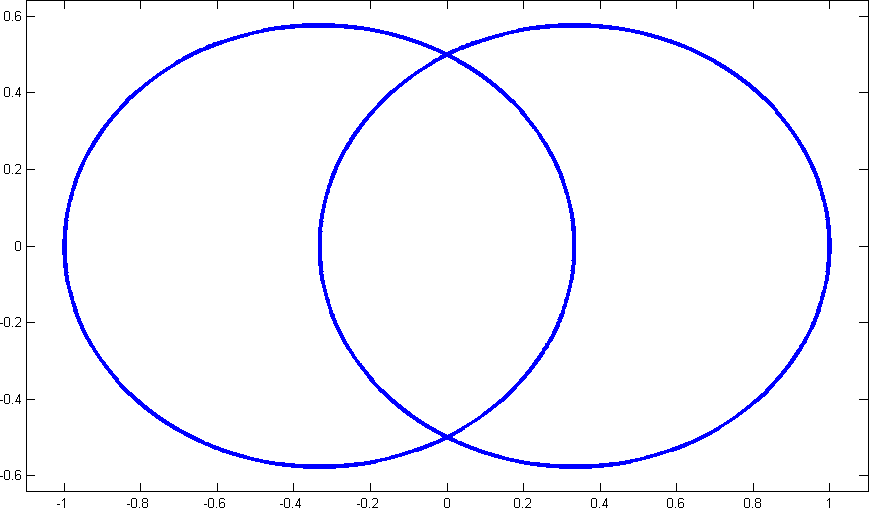}
\caption{The planar two-body problem for Section \ref{example1} in the $(x,y)$ plane.}
\label{figKeplerBase}
\end{figure}
\end{center}

This orbit has period $T \approx 2.4183$.  The functions $r_1(t)$ and $r_2(t)$ satisfy $r_1(t) = r_2(t)$ for all $t$. Further, with the initial conditions at the maximum distance from the origin, we have $r_1(T-t) = r_1(t)$.  A similar symmetry $r_1(T/2 + t) = r_1(T/2 - t)$ exists at the point where the two bodies reach their minimum distance from the origin. \\

As $R_1 = R_2 = 1$ for this orbit, we choose the value of $q_0 = Q_0 = 1$ for convenience.  Using the procedure described in Section \ref{preNumerics} and the standard Runge-Kutta integration, we find the value of $P = f(\theta)$ for an evenly-spaced grid of points in the interval $\theta \in [0, T]$.  The results of integrating these points back to the $q = 0$ plane are shown in Figure \ref{figKepler1}.  It is worth noting that the peak value for parabolic escape occurs just before $t = T/2$.  This is expected, as the gravitational pull along the $z$-axis of the planar orbit cannot be maximized at $T/2$ if $q = 0$.

\begin{center}
\begin{figure}[h]
\includegraphics[scale=.65]{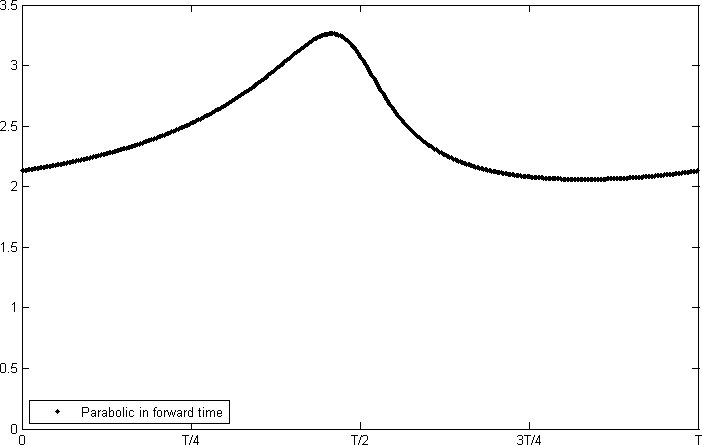}
\caption{The limit of $\mathcal{S}$ on the $q = 0$ plane approaching from $q > 0$.  Here, $\theta$ is plotted on the horizontal axis, and $p$ is plotted on the vertical.}
\label{figKepler1}
\end{figure}
\end{center}

Let $\mathcal{S}_0^+$ denote the set of points shown in Figure \ref{figKepler1}.  Owing to the time-reversing symmetry, we can also find the orbits which achieve parabolic escape in reverse time by reflecting the set $\mathcal{S}_0^+$ across the line $t = T$.  Denote the resulting set $\mathcal{S}_0^-$.  This is shown in Figure \ref{figKepler2}. 

\begin{center}
\begin{figure}[h]
\includegraphics[scale=.65]{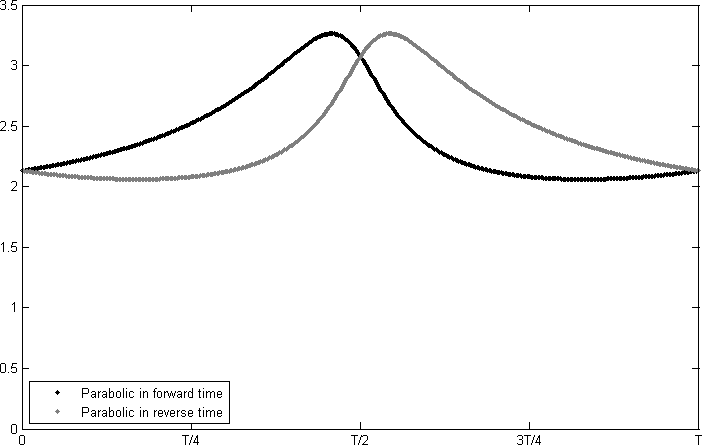}
\caption{Figure \ref{figKepler1} (darker), with the reverse-time parabolic escape orbit curve overlaid (lighter).}
\label{figKepler2}
\end{figure}
\end{center}

The set of points in the region where $\mathcal{S}_0^+$ lies below $\mathcal{S}_0^-$ $(T/2 < t < T)$ denotes initial conditions with $q = 0$ for which an orbit escapes hyperbolically in reverse time but returns at least once to the $q = 0$ plane in forward time.  Similarly, points lying below $\mathcal{S}_0^+$ but above $\mathcal{S}_0^-$ $(0 < t < T/2)$ escape hyperbolically in forward time, but eventually return to $q = 0$ in reverse time.  Points on either curve that lie above the other escape parabolically in either forward or reverse time, and hyperbolically in the other.  The intersections of the two curves, which occur at $t = 0$ and $t = T/2$, correspond to the orbits which escape parabolically in both forward and reverse time. \\

We can observe more complicated behavior by integrating points that lie below the $\mathcal{S}_0^+$ curve and observing their next intersection with the $q = 0$ plane.  Numerically, this can be slightly problematic, as behaviors near $q = \infty$ or $Q = 0$ involve high-order powers of very small terms.  However, certain points far below $\mathcal{S}_0^+$ present no such problem.  We show the results of one such integration in Figure \ref{figKepler3}.  (It is worth noting that, strictly speaking, the starred points should have $p < 0$.  However, since $\phi_t$ is symmetric with respect to  $(q,p,\theta) \mapsto (-q, -p, \theta)$, an image similar to Figure \ref{figKepler2} exists in the $q = 0$, $p < 0$ half-plane, and so we may consider the starred points in Figure \ref{figKepler3} as the image of the corresponding points with $p < 0$.)\\

\begin{center}
\begin{figure}[h]
\includegraphics[scale=.65]{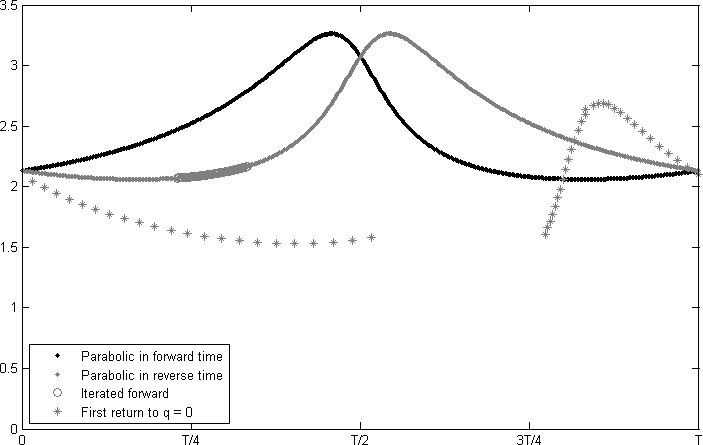}
\caption{The indicated set of points on $\mathcal{S}_0^-$ integrated until they return to the $q = 0$ plane.  The image under $\phi_t$ is denoted by the starred points.}
\label{figKepler3}
\end{figure}
\end{center}

An interesting behavior is observed at the intersection of the curve of starred points and $\mathcal{S}_0^+$.  This is another orbit that escapes parabolically in both forward and reverse time, but the number of crossings through $q = 0$ varies.  In this case, the forward-time orbit does not return to $q = 0$, but the reverse time orbit returns exactly once.  Hence, this orbit connects a parabolic escape orbit at $q = \infty$ to a parabolic escape orbit at $q = \infty$ that passes through $q = 0$ exactly twice.  In the $p < 0$ portion of the plane, the same phenomenon occurs, with an orbit connecting $q = -\infty$ to itself. \\

Again, due to the time-reversing symmetry, we obtain for free the result of the reverse-time image of the corresponding points on $\mathcal{S}_0^+$.  The results are shown in Figure \ref{figKepler4}.  Here, again, a new phenomenon occurs, with the intersection of the two starred curves, near the point $(T/2, 3/2)$.  This point crosses $q = 0$ exactly once in  forward and backward time, and then escapes parabolically.  Hence, there is a parabolic orbit connecting $q = -\infty$ to $q = \infty$ and which passes through $q = 0$ exactly thrice.   \\

\begin{center}
\begin{figure}[h]
\includegraphics[scale=.65]{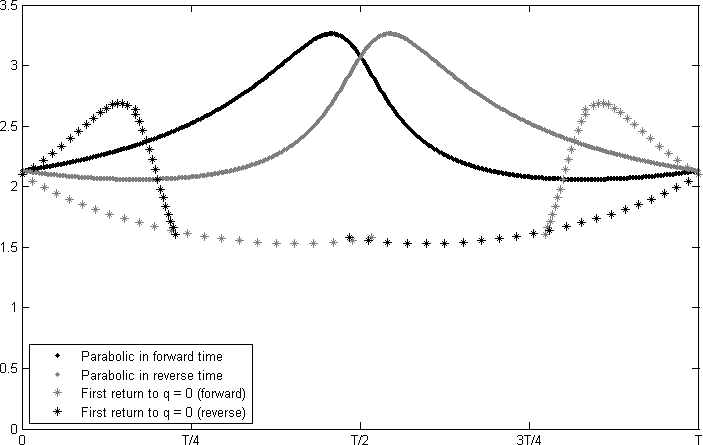}
\caption{The symmetric image showing forward-time images of some points in $\mathcal{S}_0^-$ and reverse-time images of some points in $\mathcal{S}_0^+$.}
\label{figKepler4}
\end{figure}
\end{center}

\subsection{A Configuration with Collision Singularities away from the Origin}\label{example2}

In this section, we consider the planar orbit discussed in a series of papers (\cite{bibOYS}, \cite{bibBORSY}, \cite{bibBOYS1} and \cite{bibBOYS2}, see also \cite{bibBMS}) that features simultaneous binary collisions away from the origin.  Four bodies, each of mass $m = 1$, initially lie on the coordinate axes.  Their initial velocities are perpendicular to the coordinate axes and equal in magnitude, leading to collisions as shown in Figure \ref{figSBCBase}. It is shown in \cite{bibOYS} that the orbit exists as pictured, and is symmetric through rotation through the angle $\pi$.  Moreover, the collisions in the orbit are regularizable.  The regularization of the collisions of the four bodies involves changes in spatial  coordinates, as well as a time change of the form $d\hat{t}/dt = u(x,y)$.  The net effect of all of these changes is that the velocities of the four bodies is finite in the new coordinate frame, so the orbit may be continued past collision.  (A demonstration of regularization is given in Appendix \ref{rhombAppendix}.)

\begin{center}
\begin{figure}[h]
\includegraphics[scale=.65]{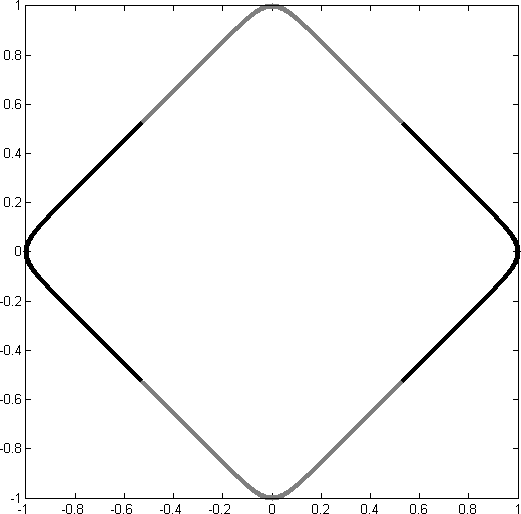}
\caption{The simultaneous binary collision orbit featured in Section \ref{example2}.  At all times, the four bodies lie at positions $(x,y)$, $(y,x)$, $(-x,-y)$, and $(-y,-x)$.  Collisions occur along the lines $y = x$ and $y = -x$.}
\label{figSBCBase}
\end{figure}
\end{center}

We may adapt the motion of the massless particle by performing the same time change on $q$ and $p$, namely:
$$\dot{\hat{q}} = \frac{d\hat{t}}{dt}\dot{q}, \quad \dot{\hat{p}} = \frac{d\hat{t}}{dt}\dot{p}.$$
Similar transformations can be made to obtain inverted coordinates $\dot{\hat{Q}}$ and $\dot{\hat{P}}$.  Then, replacing the time variable $t$ by the fictional time variable $\tau$, we obtain an equivalent system without singularities for the planar configuration and the massless particle.  Specifically, the curves traced in $\mathbb{R}^4$ by $(x(t), y(t), q(t), p(t))$ and $(x(\hat{t}), y(\hat{t}), q(\hat{t}), p(\hat{t})$ are identical as sets of points, and differ only in the parameterization.  Hence, we can perform our analysis in the regularized setting without any difficulty. \\

Placing the planar bodies at $(\pm 1, 0)$ and $(0, \pm 1)$ with appropriate initial velocities yields a planar orbit with period $\hat{t} \approx 6.4848$.  Moreover, by the symmetry shown in Figure {figSBCBase}, we have that $r_1 = r_2 = r_3 = r_4$ for all time.  For the first half of the period, we have all four bodies lying in the first and third quadrants, returning to the coordinate axes at the end of the interval.  Then, over the second half, the four bodies lie in the second and fourth coordinates, repeating the same behavior up to reflection across either coordinate axis.  Hence, each of the functions $r_i$ is periodic with period $T \approx 3.2424,$ as reflection across the coordinate axes does not change radial distance.  \\

Figure \ref{figSBC1} shows the result of repeating the numerical work as in Section \ref{example1}.  The central gap corresponds to the interval of time containing the collision.  At this time, the value of $d\hat{t}/dt$ approaches zero, lengthening a momentary $t$-interval to a longer $\hat{t}$-interval.  It is also important to note that the vertical scale is quite small, so the two curves are actually quite close to each other.  For this reason, performing further numerical integration to obtain a figure similar to Figure \ref{figKepler4} is problematic. \\

\begin{center}
\begin{figure}[h]
\includegraphics[scale=.5]{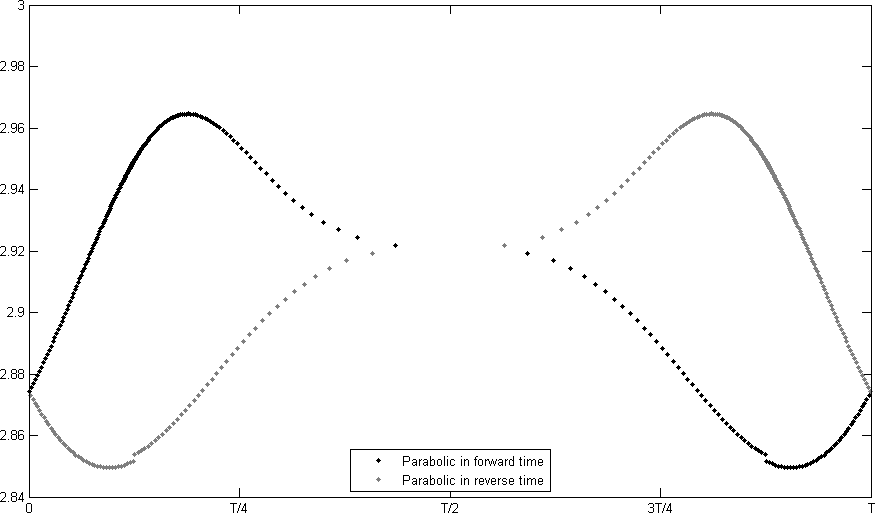}
\caption{The $q = 0$ image corresponding to Figure \ref{figKepler2} for the orbit of Section \ref{example2}.  Fictionalized time $\hat{t}$ lies on the horizontal axis, with $p$ on the vertical.  $T \approx 3.2424$.}
\label{figSBC1}
\end{figure}
\end{center}

\subsection{The $e = 1$ Sitnikov Orbit, or the Restricted Rhomboidal Problem}\label{example3}

As mentioned in the introduction, this work was the result of studying the Rhomboidal configuration as one pair of masses approaches $m = 0$.  When $m = 0$, we have a collinear two-body configuration with collisions, and a pair of massless particles that are symmetric across the collinear configuration.  Since the two massless particles have no influence over each other, removing one does not change the overall dynamics.  In this setting, we have the Sitnikov problem with $e = 1$.  For this, the solution of the planar orbit can be given explicitly, in both regularized and standard coordinates.  In regularized time, we have that
$$r_1(\hat{t}) = r_2(\hat{t}) = \left|\frac{\sqrt{2}}{2}\sin\left(\frac{\hat{t}}{2}\right)\right|.$$
(A proof of this is included as Appendix \ref{rhombAppendix}).  This is $2\pi$-periodic in $\hat{t}$.  Since $r_i = 0$ for certain values of $\hat{t}$, Theorem \ref{boundedp} does not apply.  We expect to see asymptotic ``spikes'' in the $q = 0$ plane corresponding to the times at which $r_i = 0$.  Carrying out the same numerical studies as previous sections gives the result shown in Figure \ref{figRhomb1}.

\begin{center}
\begin{figure}[h]
\includegraphics[scale=.5]{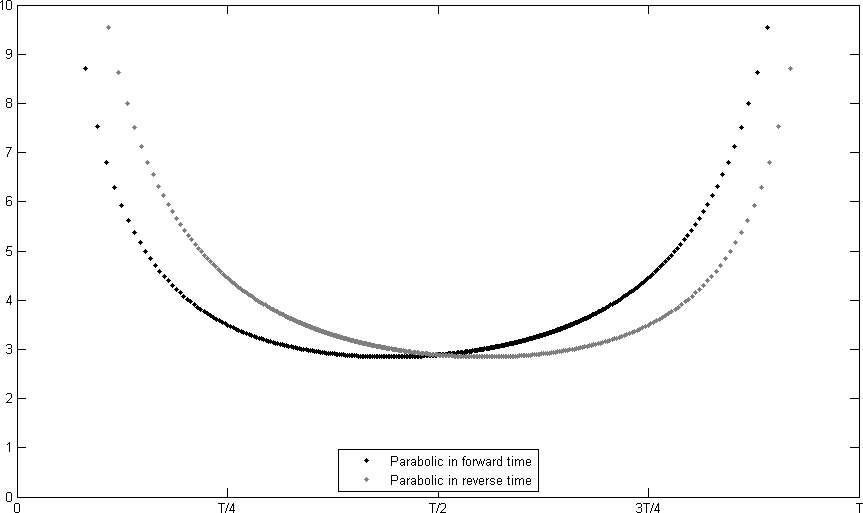}
\caption{The $q = 0$ image corresponding to Figure \ref{figKepler2} for the orbit of Section \ref{example3}.  Fictionalized time $\hat{t}$ lies on the horizontal axis, with $p$ on the vertical.  $T = 2\pi$.}
\label{figRhomb1}
\end{figure}
\end{center}

We can again produce the first-return map for the points corresponding to reverse-time parabolic orbit lying in $\theta > T/2$.  In this instance, it is numerically feasible to perform this calculation for nearly all such points.  Prior to reduction mod $T$, asymptotic spikes appear at values of $\theta = nT$ for integer values of $n$.  These are visible in Figure \ref{figRhomb2}.

\begin{center}
\begin{figure}[h]
\includegraphics[scale=.5]{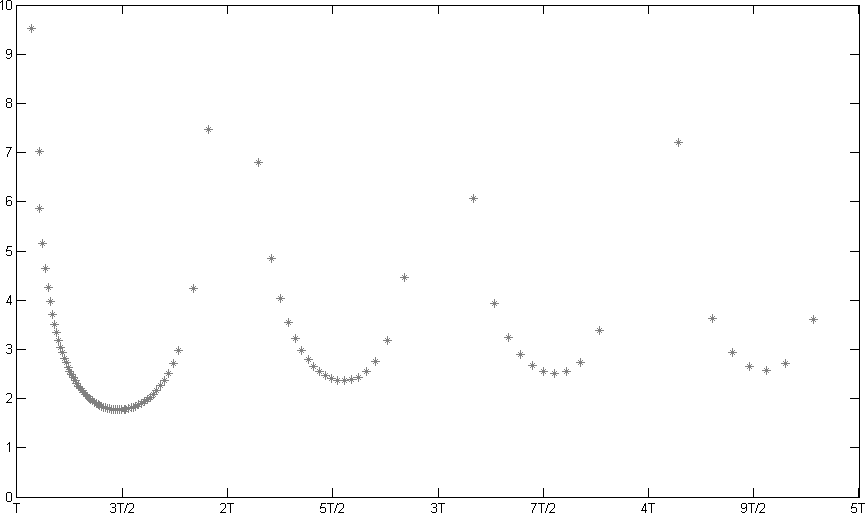}
\caption{The $q = 0$ image corresponding to forward iteration of the ``parabolic in reverse time'' points lying below the ``parabolic in forward time'' curve, prior to reduction mod $T$.  Fictionalized time $\hat{t}$ lies on the horizontal axis, with $p$ on the vertical.  $T = 2\pi$.}
\label{figRhomb2}
\end{figure}
\end{center}

If we choose the interval $\theta \in [T,2T]$ and overlay the corresponding points with Figure \ref{figRhomb1} (along with their reverse-time counterparts), we obtain the result shown in Figure \ref{figRhomb3}.  We could, of course, overlay any of the intervals $[nT, (n+1)T]$ and get a similar picture.  In this case, the intersection of the return curves would correspond to orbits which escape parabolically in forward and reverse time, and cross through the $q = 0$ plane three times.  By appropriate choice of $n$, we could specify an arbitrary number of periods through which the planar orbit passes between each of these three returns.

\begin{center}
\begin{figure}[h]
\includegraphics[scale=.5]{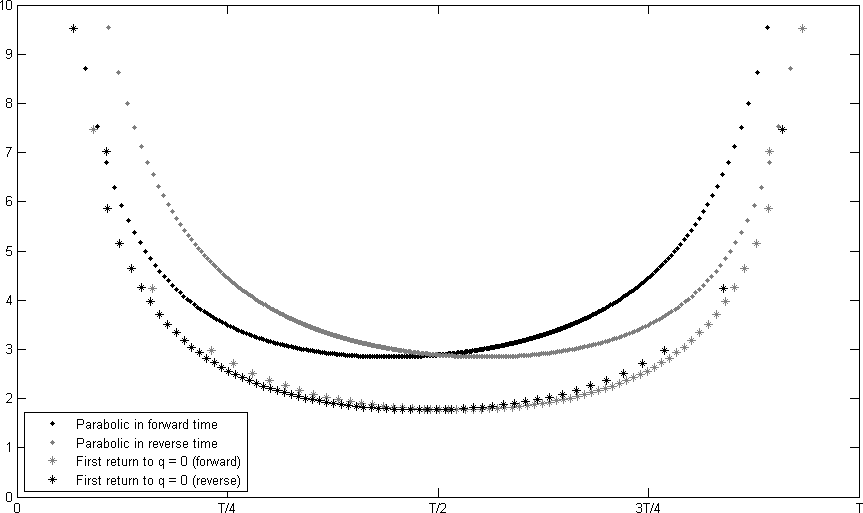}
\caption{Analogue of Figure \ref{figKepler4} for the orbit of Section \ref{example3}.  $T = 2\pi$.}
\label{figRhomb3}
\end{figure}
\end{center}

We could continue in similar fashion to obtain connecting parabolic orbits passing through $q = 0$ arbitrarily many times, with any finite integer sequence of periods completed between returns.  It may be the case that the return times do not occur at exact integer multiples of the period, but counting only completed periods (e.g. counting collisions in the case of Sections \ref{example2} and \ref{example3}) provides a more precise interpretation.

\section{Concluding Remarks}\label{concluding}
\subsection{Future Research}
There are a few questions that we believe warrant further exploration, which we have chosen not to address in this paper.  The foremost is the following:
\begin{question}
Is $\mathcal{S}$ anything more than $C^0$?  To what degree is this connected to the functions $r_i$?
\end{question}
For example, the functions $r_i$ of Sections \ref{example2} and \ref{example3} are merely continuous, with cusps at each collision time, whereas those of Section \ref{example1} are at least $C^\infty$.  Intuitively, we would expect this to cause a difference in the resulting surfaces $\mathcal{S}$ corresponding to each orbit.

\begin{question}
Do the symbolic dynamics discussed in Section \ref{example3} hold for all orbits?
\end{question}
We would expect that near the points in the $q = 0$ plane where forward- and reverse-time parabolic orbits intersect that there are return orbits with arbitrarily long return times.  Since the planar orbits are periodic, this gives that an arbitrary number of periods can occur before the massless particle returns to the origin.  The precise details of how these return regions overlap, however, will require much more work.

%\subsection{Acknowledgements}
%Lorem ipsum dolor sit amet, consectetur adipiscing elit. Aenean lectus nisi, elementum vitae tempor imperdiet, sodales at diam. Duis pharetra fermentum commodo. Nulla eget augue scelerisque, euismod elit ut, interdum massa. Vestibulum nec neque at eros blandit vestibulum. Ut id scelerisque ex. Proin eget ex a ligula volutpat finibus et tempor nisi. Proin dapibus pulvinar ante, in laoreet quam venenatis mollis. 

\appendix
\section{Regularization of Collisions in the 2-Body Collinear Problem} \label{rhombAppendix}

Consider the two-body collinear configuration with two bodies of mass $1$ located at $(\pm x, 0)$.  Their conjugate momenta is given by $x = 2\dot{y}$.  The Hamiltonian in this setting is then given by
$$H = \frac{1}{4}y^2 - \frac{1}{2x}.$$
Collisions at the origin can be regularized via
$$x = X^2, \quad Y = 2xy.$$
Solving for $y$ and substituting yields
\begin{align*}
\tilde{H} &= \frac{1}{4}\left(\frac{Y^2}{4X^2}\right) - \frac{1}{2X^2} \\
&= \frac{Y^2}{16X^2} - \frac{1}{2X^2}
\end{align*}
Setting $dt/ds = Q_1^2$, the final regularized Hamiltonian in extended phase space is given by
\begin{align*}
\Gamma &= \frac{dt}{ds}(\tilde{H} - E) \\
&= X^2\left(\frac{Y^2}{16X^2} - \frac{1}{2X^2} \right) \\
&= \frac{Y^2}{16} - \frac{1}{2} - X^2E.
\end{align*}
At collision ($X = 0$), assuming $\Gamma = 0$, we find $Y = \sqrt{8}$.  The collision has therefore been regularized as claimed.\\

The equations of motion for this system can be solved explicitly.  We find
$$\dot{X} = \frac{\partial \Gamma}{\partial Y} = \frac{Y}{8},  \qquad \dot{Y} = -\frac{\partial \Gamma}{\partial X} = 2XE.$$
Differentiating the expression for $\dot{Q_1}$ and substituting gives
$$\ddot{X} = \frac{\dot{Y}}{8} = \frac{YE}{4}.$$
Assuming $E = -1$, the differential equation
$$\ddot{X} + \frac{X}{4} = 0$$
has a solution of the form
$$X(s) = c_1\cos(s/2) + c_2\sin(s/2).$$
The corresponding solution for $Y(s)$ is then
$$Y(s) = -4c_1\sin(s/2) + 4c_2\cos(s/2).$$

Assuming that the bodies start at collision ($X(0) = 0$), we find that $c_1 = 0.$  At that time, $Y(0) = \sqrt{8}$, and so $c_2 = \sqrt{2}/2$. Then dynamically significant events for the two primaries occur at $s = n\pi$.  In particular, collisions occur when $n$ is even, and maximum displacement occurs when $n$ is odd.\\

For more complicated orbits (such as those of Section \ref{example2}, additional coordinate transformations may need to take place before the regularization can be completed.  We refer the reader to the papers listed in that section for complete details.

\newpage

\bibliographystyle{plain}
\bibliography{Separation}

\begin{thebibliography}{10}

\bibitem{Alekseev3}
V.~M. Alekseev.
\newblock Quasirandom dynamical systems. {I}. {Q}uasirandom diffeomorphisms.
\newblock {\em Mat. Sb. (N.S.)}, 76 (118):72--134, 1968.

\bibitem{Alekseev1}
V.~M. Alekseev.
\newblock Quasirandom dynamical systems. {II}. {O}ne-dimensional nonlinear
  vibrations in a periodically perturbed field.
\newblock {\em Mat. Sb. (N.S.)}, 77 (119):545--601, 1968.

\bibitem{Alekseev2}
V.~M. Alekseev.
\newblock Quasirandom dynamical systems. {III}. {Q}uasirandom vibrations of
  one-dimensional oscillators.
\newblock {\em Mat. Sb. (N.S.)}, 78 (120):3--50, 1969.

\bibitem{bibBS1}
Lennard Bakker and Skyler Simmons.
\newblock Stability of the rhomboidal symmetric-mass orbit.
\newblock {\em To appear in Disc. Cont. Dyn. Sys. A}, 35(1), 2015.

\bibitem{bibBMS}
Lennard~F. Bakker, Scott Mancuso, and Skyler~C. Simmons.
\newblock Linear stability for some symmetric periodic simultaneous binary
  collision orbits in the planar pairwise symmetric four-body problem.
\newblock {\em J. Math. Anal. Appl.}, 392(2):136--147, 2012.

\bibitem{bibBOYS1}
Lennard~F. Bakker, Tiancheng Ouyang, Duokui Yan, and Skyler Simmons.
\newblock Existence and stability of symmetric periodic simultaneous binary
  collision orbits in the planar pairwise symmetric four-body problem.
\newblock {\em Celestial Mech. Dynam. Astronom.}, 110(3):271--290, 2011.

\bibitem{bibBOYS2}
Lennard~F. Bakker, Tiancheng Ouyang, Duokui Yan, and Skyler Simmons.
\newblock Erratum to: {E}xistence and stability of symmetric periodic
  simultaneous binary collision orbits in the planar pairwise symmetric
  four-body problem [mr2821623].
\newblock {\em Celestial Mech. Dynam. Astronom.}, 112(4):459--460, 2012.

\bibitem{bibBORSY}
Lennard~F. Bakker, Tiancheng Ouyang, Duokui Yan, Skyler Simmons, and Gareth~E.
  Roberts.
\newblock Linear stability for some symmetric periodic simultaneous binary
  collision orbits in the four-body problem.
\newblock {\em Celestial Mech. Dynam. Astronom.}, 108(2):147--164, 2010.

\bibitem{bibFaruque1}
S.~B. Faruque.
\newblock Solution of the {S}itnikov problem.
\newblock {\em Celestial Mech. Dynam. Astronom.}, 87(4):353--369, 2003.

\bibitem{bibHagel1}
J.~Hagel.
\newblock A new analytic approach to the {S}itnikov problem.
\newblock {\em Celestial Mech. Dynam. Astronom.}, 53(3):267--292, 1992.

\bibitem{bibHagelLhotka1}
Johannes Hagel and Christoph Lhotka.
\newblock A high order perturbation analysis of the {S}itnikov problem.
\newblock {\em Celestial Mech. Dynam. Astronom.}, 93(1-4):201--228, 2005.

\bibitem{bibJimenezEscalona1}
Lidia Jim{\'e}nez-Lara and Adolfo Escalona-Buend{\'{\i}}a.
\newblock Symmetries and bifurcations in the {S}itnikov problem.
\newblock {\em Celestial Mech. Dynam. Astronom.}, 79(2):97--117, 2001.

\bibitem{bibLiuSun1}
Jie Liu and Yi~Sui Sun.
\newblock On the {S}itnikov problem.
\newblock {\em Celestial Mech. Dynam. Astronom.}, 49(3):285--302, 1990.

\bibitem{bibLlibreOrtega1}
Jaume Llibre and Rafael Ortega.
\newblock On the families of periodic orbits of the {S}itnikov problem.
\newblock {\em SIAM J. Appl. Dyn. Syst.}, 7(2):561--576, 2008.

\bibitem{bibMacMillan1}
W.~MacMillan.
\newblock An integrable case in the restricted problem of three bodies.
\newblock {\em Astron. J.}, 27:11--13, 1913.

\bibitem{bibMarchesinCastilho1}
Marcelo Marchesin and C{\'e}sar Castilho.
\newblock Subharmonic solutions in the {S}itnikov problem.
\newblock {\em Qual. Theory Dyn. Syst.}, 7(1):213--226, 2008.

\bibitem{bibMarchesinVidal1}
Marcelo Marchesin and Claudio Vidal.
\newblock Spatial restricted rhomboidal five-body problem and horizontal
  stability of its periodic solutions.
\newblock {\em Celestial Mech. Dynam. Astronom.}, 115(3):261--279, 2013.

\bibitem{bibMartinez}
Regina Mart{\'{\i}}nez.
\newblock On the existence of doubly symmetric ``{S}chubart-like'' periodic
  orbits.
\newblock {\em Discrete Contin. Dyn. Syst. Ser. B}, 17(3):943--975, 2012.

\bibitem{bibMcGehee1}
Richard McGehee.
\newblock A stable manifold theorem for degenerate fixed points with
  applications to celestial mechanics.
\newblock {\em J. Differential Equations}, 14:70--88, 1973.

\bibitem{Moser1}
J{\"u}rgen Moser.
\newblock {\em Stable and random motions in dynamical systems}.
\newblock Princeton University Press, Princeton, N. J.; University of Tokyo
  Press, Tokyo, 1973.
\newblock With special emphasis on celestial mechanics, Hermann Weyl Lectures,
  the Institute for Advanced Study, Princeton, N. J, Annals of Mathematics
  Studies, No. 77.

\bibitem{bibOrtegaRivera1}
Rafael Ortega and Andr{\'e}s Rivera.
\newblock Global bifurcations from the center of mass in the {S}itnikov
  problem.
\newblock {\em Discrete Contin. Dyn. Syst. Ser. B}, 14(2):719--732, 2010.

\bibitem{bibOYS}
Tiancheng Ouyang, Duokui Yan, and Skyler Simmons.
\newblock Periodic solutions with singularities in two dimensions in the
  $n$-body problem.
\newblock {\em Rocky Mtn. J. Math.}, 42(4):1601--1614, 2012.

\bibitem{bibPerdiosMarkellos}
E.~Perdios and V.~V. Markellos.
\newblock Stability and bifurcations of {S}itnikov motions.
\newblock {\em Celestial Mech.}, 42(1-4):187--200, 1987/88.

\bibitem{bibPerdios1}
E.~A. Perdios.
\newblock The manifolds of families of 3{D} periodic orbits associated to
  {S}itnikov motions in the restricted three-body problem.
\newblock {\em Celestial Mech. Dynam. Astronom.}, 99(2):85--104, 2007.

\bibitem{bibRivera1}
Andr{\'e}s Rivera.
\newblock Periodic solutions in the generalized {S}itnikov {$(N+1)$}-body
  problem.
\newblock {\em SIAM J. Appl. Dyn. Syst.}, 12(3):1515--1540, 2013.

\bibitem{bibShib1}
Mitsuru Shibayama.
\newblock Minimizing periodic orbits with regularizable collisions in the
  {$n$}-body problem.
\newblock {\em Arch. Ration. Mech. Anal.}, 199(3):821--841, 2011.

\bibitem{bibSidorenko1}
Vladislav~V. Sidorenko.
\newblock On the circular {S}itnikov problem: the alternation of stability and
  instability in the family of vertical motions.
\newblock {\em Celestial Mech. Dynam. Astronom.}, 109(4):367--384, 2011.

\bibitem{bibSitnikov1}
K.~Sitnikov.
\newblock The existence of oscillatory motions in the three-body problems.
\newblock {\em Soviet Physics. Dokl.}, 5:647--650, 1960.

\bibitem{bibSoulisPapadakisBountis}
P.~S. Soulis, K.~E. Papadakis, and T.~Bountis.
\newblock Periodic orbits and bifurcations in the {S}itnikov four-body problem.
\newblock {\em Celestial Mech. Dynam. Astronom.}, 100(4):251--266, 2008.

\bibitem{bibWaldvogel1}
J\"org Waldvogel.
\newblock The rhomboidal symmetric four-body problem.
\newblock {\em Celestial Mech. Dynam. Astronom.}, 113(1):113--123, 2012.

\bibitem{bibYan1}
Duokui Yan.
\newblock Existence and linear stability of the rhomboidal periodic orbit in
  the planar equal mass four-body problem.
\newblock {\em J. Math. Anal. Appl.}, 388(2):942--951, 2012.

\end{thebibliography}

\end{document}